\documentclass[12pt]{article}
\usepackage[a4paper]{geometry}
\usepackage{amsthm}

\usepackage{amssymb}
\usepackage{amsmath}
\usepackage{eufrak}

\newtheorem{theorem}{Theorem}

\newtheorem{corollary}[theorem]{Corollary}
\newtheorem{lemma}[theorem]{Lemma}

\newtheorem{question}{Question}
\newtheorem{proposition}[theorem]{Proposition}

\begin{document}
\def\F{{\mathbb F}}
\title{ Algebraic approach to Rump's results on relations between  braces and pre-Lie algebras }

\author{Agata Smoktunowicz}
\date{}
\maketitle
\begin{abstract}

In 2014, Wolfgang  Rump showed that there exists a correspondence  between left nilpotent right $\mathbb R$-braces and pre-Lie algebras.
This correspondence, established  using a geometric approach related to flat affine manifolds and affine torsors, works locally. In this paper we explain  Rump's correspondence using only algebraic formulae. 
An algebraic interpretation  of the  correspondence works for fields of sufficiently large prime characteristic as well as for fields of characteristic zero. 

\end{abstract}

\section{Introduction}

In \cite{Rump}, Wolfgang Rump showed that there exists a correspondence  between left nilpotent right $\mathbb R$-braces and pre-Lie algebras.   Braces were introduced by Rump in 2005 to describe all involutive set-theoretic solutions of the Yang-Baxter equation. This approach subsequently found applications in several other research areas. 

  The  algebraic formula for Rump's  connection from strongly nilpotent braces to pre-Lie algebras is particularly simple and enables  rapid production of  examples of pre-Lie algebras from braces.

\begin{question}
 Is there a formula for a  passage from finite pre-Lie algebras to fnite $\mathbb F$-braces, where $F$ is a finite field?
\end{question}

 The algebraic formula for a passage from pre-Lie algebras to $\mathbb F$-braces in Rump's correspondence is not easy to use in examples, which raises the question:
\begin{question}
 Is there an easy  way  to obtain braces from pre-Lie algebras?
\end{question}

 In this paper all considered braces are left braces. Recall that if we change the multiplication in a left brace to the opposite multiplication we will obtain a right brace, and vice-versa. In  Rump's papers \cite{rump}, \cite{Rump},  a brace always means a right brace. In  \cite{Rump} Rump presented a correspondence between right  $\mathbb R$-braces and RSA-the right symmetric algebras. By taking the opposite multiplication in both the brace and the RSA it gives the correspondence between
 left-braces and LSA, left symmetric algebras, also called pre-Lie algebras.

 Notice that it is known that pre-Lie algebras are in correspondence with the {\' e}tale  affine representations of nilpotent Lie algebras \cite{B2}, and with Lie algebras with 1-cocycles \cite{Rump},\cite{Bai}, \cite{Burde}, they also appear in noncommutative geometry. 
Braces have also been linked to other
research areas, for example, in \cite{gateva}, Gateva-Ivanova showed that there is a correspondence between
braces and braided groups with an involutive braiding operator, whereas in \cite{DB}, Bachiller observed that
there is a connection between braces and Hopf-Galois extensions of abelian type (see
also the appendix to \cite{SVB} for some further results). In \cite{Rump, DB, Catino, gateva, SVB} braces and skew braces have been shown to be equivalent  to several concepts in group theory (1-cocycles, regular subgroups, matched pairs of groups). Moreover, two-sided braces are exactly
the Jacobson radical rings \cite{rump}, \cite{cjo}. In \cite{doikou}, applications of braces in quantum integrable
systems were investigated, and in \cite{Agatka1} R-matrices constructed from braces were studied.
Solutions of the reflection equation related to braces have also been investigated by several
authors.

In this paper we use purely algebraic methods to look closely at Rump's correspondence between braces and pre-Lie algebras and give algebraic formulas for this correspondence.

\section{Background information}

 A pre-Lie algebra $A$ is a vector space with a binary operation $(x, y) \rightarrow  xy$
satisfying
\[(xy)z -x(yz) = (yx)z - y(xz),\]

 for every $x,y,z\in A$. We say that a pre-Lie algebra $A$  is nilpotent if, for some $n$, all products of $n$ elements in $A$ are zero.

Recall that a set $B$ with binary operations $+$ and $* $ is a {\em  left brace} if $(A, +)$ is an abelian group and the following version of distributivity combined with associativity holds:
  \[(a+b+a*b)* c=a* c+b* c+a* (b* c), a* (b+c)=a* b+a* c,\]
 moreover  $(B, \circ )$ is a group where we define $a\circ b=a+b+a* b$.

 See \cite{rump} for the original definition. For a shorter equivalent  definition using group theory see \cite{cjo}.
In what follows we will use the definition in terms of operation '$\circ $' ( presented in \cite{cjo}): a set $B$ with binary operations of addition $+$, and multiplication $\circ $ is a brace if $(A, +)$ is an abelian group, $(A, \circ )$ is a group and for every $a,b,c\in A$
\[a\circ (b+c)+a=a\circ b+a\circ c.\]

 In \cite{Catino}, Catino and Rizzo defined
$\mathbb F$-braces thus:
 Let $\mathbb F$ be a field, let $(B, +, \circ )$ be a brace, then $B$ is an $\mathbb F$-brace if $a*(eb)=e(a* b)$ for all $a,b\in B,$
$ e\in \mathbb F$. Here $ a*b=a \circ b -a -b$.

In \cite{rump} Rump introduced left and right nilpotent braces and radical chains $A^{i+1}=A*A^{i}$ and $A^{(i+1)}=A^{(i)}*A$,  for a left brace $A$, where  $A=A^{1}=A^{(1)}$ (the original construction of Rump is for right braces, but we give the natural translation of it to left braces here). Recall that a left brace $A$  is left nilpotent if  there is a number $n$ such that $A^{n}=0$, where inductively $A^{i}$ consists of sums of elements $a*b$ with
$a\in A, b\in A^{i-1}$. A left brace $A$ is right nilpotent   if  there is a number $n$ such that $A^{(n)}=0$, where $A^{(i)}$ consists of sums of elements $a*b$ with
$a\in A^{(i-1)}, b\in A$.
 Strongly nilpotent braces and the chain of ideals $A^{[i]}$ of a brace $A$ were defined in \cite{Engel}.
 Define $A^{[1]}=A$.  A left brace $A$ is strongly nilpotent if  there is a number $n$ such that $A^{[n]}=0$, where $A^{[i]}$ consists of sums of elements $a*b$ with
$a\in A^{[j]}, b\in A^{[i-j]}$ for all $0<j<i$.
A   brace is strongly nilpotent if and only if it is both left nilpotent and right nilpotent \cite{Engel}.

  Various other radicals in braces were subsequently introduced, in analogy with ring theory and group theory. Recall that solvable braces were introducted in \cite{djc}, and in \cite{ksv} the  connection between prime radical and  solvable braces was investigated. See also \cite{kinneart} for some further results on solvable braces.
In \cite{JespersLeandro} the radical of a brace was  introduced as the intersection of all maximal ideals in a given brace. This radical enjoys good properties, and is very useful for describing the structure of a given brace. 

\section{Passage from pre-Lie algebras to braces}\label{mi}

In \cite{AG}, A.Agrachev and R. Gamkrelidze introduced the formal group of flows constructed from a pre-Lie algebra. Notice that this group of flows combined with the same addition is a brace, and it is the same brace as obtained in Rump's correspondence. As mentioned by Rump in a private correspondence the addition in the pre-Lie algebra and in the corresponding brace is always the same. In his survey \cite{M},  D. Manchon mentions this group of flows along with explanations of their structure. To summarise  from \cite{AG}, \cite{M}, let $A$ with operations $\cdot, +$ be a pre-Lie algebra over $\mathbb F$, so
\[a\cdot (b\cdot c)-(a\cdot b)\cdot c=b\cdot (a\cdot c)-(b\cdot a)\cdot c.\]

Following the Rump correspondence \cite{Rump}, define the $\mathbb F$-brace $(A,+, \circ )$ with the same addition as in pre-Lie algebra $A$ and with the multiplication $\circ $ defined  as in the group of flows as follows. The following is based on \cite{AG}, \cite{M}, \cite{Rump}. We additionally assume that $A$ is a nilpotent pre-Lie algebra, and that $\mathbb F$ is a field of characteristic zero (or of characteristic larger than the nilpotency index of $A$). We use notation from \cite{M}.
\begin{enumerate}

\item Let $a\in A$, and let  $L_{a}:A\rightarrow A$ denote the left multiplication by $a$, so
$L_{a}(b)=a\cdot b$.
 Define $L_{c}\cdot L_{b}(a)=L_{c}(L_{b}(a))=c\cdot (b\cdot a)$.
 Define \[e^{L_{a}}(b)=b+a\cdot b+{\frac 1{2!}}a\cdot (a\cdot b)+{\frac 1{3!}}a\cdot (a\cdot (a\cdot b))+\cdots \]

\item  We can formally consider element $1$ such that $1\cdot a=a\cdot 1=a$ in our pre-Lie algebra (as in \cite{M})  and
 define \[W(a)=e^{L_{a}}(1)-1=a+{\frac 1{2!}}a\cdot a+{\frac 1{3!}}a\cdot (a\cdot a)+ \cdots \]
 Notice that $W(a):A\rightarrow A$ is a bijective function, provided that $A$ is a nilpotent pre-Lie algebra.

\item Let $\Omega (a):A\rightarrow A$ be the inverse function to the function $W(a)$, so  $\Omega (W(a))=W(\Omega (a))=a$.
  Following \cite{M} the first terms of $\Omega $ are
{\bf
\[ \Omega (a)=  a-{\frac 12}a\cdot a +{\frac 14} (a\cdot a)\cdot a +{\frac {1}{12}}a\cdot (a\cdot a) +\ldots \]
}

\item Define\[a\circ b=a+e^{L_{\Omega (a)}}(b).\]
 Here, the addition is the same as in the pre-Lie algebra $A$.
It was shown in \cite{AG}  that $(A, \circ )$ is a group. It is immediate to see that $(A, \circ , + )$ is a left brace because
 \[a\circ (b+c)+a=a+e^{L_{\Omega (a)}}(b+c)+a=(a+e^{L_{\Omega (a)}}(b))+(a+e^{L_{\Omega (a)}}(c))=a\circ b+a\circ c.\]

Notice that the above correspondence works globally provided that $A$ is a nilpotent pre-Lie algebra, so  $A^{n}=0$
 for some $n$.

\end{enumerate}

When the underlying pre-Lie algebra is a ring the obtained brace is with the  familiar  multiplication
 $a\circ b=a+b+ab$  (see \cite{M}).

$ $

{\em Remark about connections with the BCH formula. } Notice that the above formula can also be written using the Baker-Campbell-Hausdorff formula and  Lazard's correspondence, see \cite{AG}, \cite{M} for details.
  In particular, the following formula for the multiplication $\circ $ holds in the brace obtained above:
\[ W(a)\circ W(b)= W(C(a, b)),\]
 where $C(a, b)$ is obtained using the Campbel-Hausdorf series in the Lie algebra $L(A)$.  Recall that the Lie algebra $L(A)$ is obtained from a pre-Lie algebra $A$ by taking $[a,b]=a\cdot b-b\cdot a$, and has the same addition as  $A$.
 By the Baker-Campbell-Hausdorff formula the element $C(a, b)$ can be represented in the form of a series in variables $a, b$, multiplication by scalars and commutation in the Lie algebra $L(A)$, and 
\[C(a,b)=a+b++{\frac 12}[a,b]+{\frac 1{12}}([a,[a,b]]+[b,[b,a]])+\cdots .\] 
 For more details see \cite {AG}, page 1658, \cite{M}, page 3.

$ $

{\bf Example.} Let $(A, +, \cdot )$ be a pre-Lie algebra such that $A^{[4]}=0$.
 We calculate the formula for the multiplication in the corresponding brace $(A, +, \circ )$.

We know that  $\Omega (a)=  a-{\frac 12}a\cdot a  +c$ for some $ c\in A^{[3]}$  following the formula from \cite{M}.
 We obtain:
\[e^{L_{\Omega (a)}}(b)=b+\Omega (a)\cdot b+{\frac 12}\Omega (a)(\Omega (a)\cdot b)+c' ,\] where $c'\in A^{[4]}$, so $c'=0$.
 Therefore
\[a\circ b=a+b+a\cdot b-{\frac 12}(a\cdot a)\cdot b+{\frac 12} a\cdot (a\cdot b) \]
hence
\[a* b=a\cdot b-{\frac 12}(a\cdot a)\cdot b+{\frac 12} a\cdot (a\cdot b) \]

 Observe that the following result follows from the above construction.
\begin{theorem}\label{1}
 Let $(A, +, \cdot )$ be a nilpotent  pre-Lie algebra over a field $\mathbb F$ of characteristic zero, and let $(A, +, \circ )$ be the obtained brace as above.  Denote $a*b=a\circ b-a-b $.
 Then \[a* b=a\cdot b +\sum _{x\in B} \alpha _{x}x\] where $\alpha _{x}\in {\mathbb F}$  and $B$ is the set of all products of elements $a$ and $b$ from $(A, \cdot )$ with $b$ appearing only at the end, and $a$ appearing at least two times in each product.
\end{theorem}
\begin{proof}
 This follows immediately from the construction of $\Omega (a)$, which is a sum of $a$ and a  linear combination of all possible products of more than one element $a$ with any distribution of brackets, which can be proved by induction.
\end{proof}

\begin{question}
What braces are obtained from the  known types of pre-Lie algebras?
\end{question}

\begin{question}
What braces are obtained from Novikov algebras?
\end{question}

$ $

\section{Some supporting lemmas}

We recall  Lemma 15 from \cite{Engel}:

\begin{lemma}\label{fajny}
 Let $s$ be a natural number and let $(A, +, \circ)$ be a left brace such that $A^{s}=0$ for some $s$.
 Let $a, b\in A$, and as usual define $a*b=a\circ b-a-b$.
Define inductively elements $d_{i}=d_{i}(a,b), d_{i}'=d_{i}'(a, b)$  as follows:
$d_{0}=a$, $d_{0}'=b$, and for $i\leq 1$ define $d_{i+1}=d_{i}+d_{i}'$ and $d_{i+1}'=d_{i}d_{i}'$.
 Then for every $c\in A$ we have
\[(a+b)*c=a*c+b*c+\sum _{i=0}^{2s} (-1)^{i+1}((d_{i}*d_{i}')*c-d_{i}*(d_{i}'*c)).\]
\end{lemma}

{\bf Notation 1.} Let $A$ be a strongly nilpotent  brace with operations $+, \circ \*$. Let $x, y\in A$. Consider elements $x*y$, $x*(x*y)$, $\ldots $
 which are all products (with any distribution of brackets) of some non-zero number of  elements $x$ and one element $y$ at the end.
 The set of all such elements will be denoted as $E_{x,y}$. Notice that this set is finite, because $A$ is a strongly nilpotent brace. 
 We can list elements from the set $E_{x,y}$ in a such way that shorter products always appear before longer products, and then we can make it into a vector, which we will denote as $V_{x,y}$.

 We will now prove a supporting  lemma which will be useful in Section \ref{4}. In what follows, by $2^{n}c$ we denote the sum of $2^{n}$ copies of element $c\in A$.

\begin{lemma}\label{newest}
 Let $(A, +, \circ )$ be a  strongly  nilpotent  brace over the field of rational numbers. Then for every
$a, b\in A$ the limit $lim _{n\rightarrow \infty }2^{n}({\frac 1{2^{n}}} a)*b$ exists.

$ $ 
Moreover, there are  square matrices $P$, $T$, not depending on $a,b$, such that 

\[2^{n}({\frac 1{2^{n}}} a)*b=E_{1}PT^{n}P^{-1}V_{a,b},\]
 where $T$ is a matrix in the Jordan block form with the first Jordan block of dimension $1$  equal to $1$ and all other Jordan blocks with eigenvalues smaller than $1$. Moreover $E_{1}=[1, 0, 0, \ldots ,0]$ and entries of $S, T$ are rational numbers and
\[2^{n}V_{{\frac 1{2^{n}}} a, b}=PT^{n}P^{-1}V_{a,b}.\] 
\end{lemma}

\begin{proof}  Let notation for $E_{x,y}$ and $V_{x,y}$ be as in Notation $1$ above.
By Lemma \ref{fajny} (applied several times) every element from the set $E_{2x, y}$ can be written as a linear combination 
   of elements from $E_{x,y}$, with coefficients which do not  depend on $x$ and $y$. We can then  organise these coefficients in a matrix, which we will call $M=\{m_{i,j}\}$,  so that we obtain
\[MV_{x,y}=V_{2x,y}.\]

Notice that  elements from $E_{x,y}$ (and from $E_{2x,y}$) which are shorter appear before elements which are longer in our vectors $V_{x,y}$ and $V_{2x, y}$. Therefore by Lemma \ref{fajny} it follows that $M$ is an  upper diagonal matrix. Observe that the first element in the  vector  
 $E_{2x,y}$ is $(2x)*y$  and that this element can be written as $2(x*y)$ plus  elements of degree larger than $2$ (by  Lemma \ref{fajny}). It follows that the first diagonal entry in $M$ equals $2$, so $m_{1, 1}=2$. 
Observe that the following diagonal entries will be equal to $4$ or more, because 
 for example $(2x)*((2x)*y)$ can be written using Lemma \ref{fajny} as $4x*(x*y)$ plus  elements of degree larger than $3$.

Therefore $M$ has exactly one eigenvalue equal to $2$ with exactly one corresponding eigenvector,  and all other eigenvalues 
 equal $2^{i}$ for some $i>1$ (because diagonal entries of $M$ are its eigenvalues).
 Notice that $M$ does not depend on $x$ and $y$, as we only used relations from Lemma \ref{fajny} to construct it.
 We can write $M=PJP^{-1}$ where $J$ consists of Jordan blocks, and the first block is the $1\times 1$  diagonal entry $J_{1,1}=2$. Notice that since the eigenvalues are real it is possible to find such matrices $P, J$ with entries from the field of rational numbers. 

It follows that for every $n$, $M^{n}V_{x,y}=V_{2^{n}x,y}$, therefore 
\[2^{n}V_{2^{-n}x,y}=(2M^{-1})^{n}V_{x,y}.\]
Notice that $2M^{-1}=PJ'P^{-1}$ where $J'=2J^{-1}$ is the matrix with Jordan blocks, and the first block is $1\times 1$ block with eigenvalue $1$. The remaining blocks have eigenvalues $2^{-i}$ for $i>0$.
  It follows that we can define the limit 
\[lim_{n\rightarrow \infty } 2^{n}V_{2^{-n}x,y}=lim_{n\rightarrow \infty }PJ'^{n}P^{-1}V_{x,y}=PE_{1,1}P^{-1},\]
 where $E_{1,1}$ is the matrix with the first entry equal to $1$ and all other entries equal to zero.
 The first entry of vector $V_{2^{-n}x,y}$ is $(2^{-n}x)*y$, therefore 
  $lim _{n\rightarrow \infty } 2^{n}(2^{-n}x)*y=p_{1,1}R_{1}V_{x,y}$, where $p_{1,1}$ is the first diagonal entry of matrix $P$, and $R_{1}$ is the first row of matrix $P^{-1}$.
\end{proof}

{\bf Notation 2.} Let $A$ be a brace with operations $+, \circ , *$ defined as usual so $x\circ y=x+y+x*y$.
  For $x, y, z\in A$  and 
let $E(x, y, z)\subseteq A$ denote the set consisting of any product of elements $x$ and $y$ and one element $z$ at the end of each product under the operation $*$,  in any order, with any distribution of brackets, each product consisting of at least 2 elements from the set $\{x,y\}$, each product having $x$ and $y$ appear at least once,  and having element $z$ at the end. Notice that $E(x,y,z )$ is  finite provided that $A$ is a strongly nilpotent brace. Let $V_{x,y,z}$ be  a vector obtained from products of elements $x, y, z$ arranged in a such way that shorter products of elements  are situated  before longer products.

$ $

\begin{lemma}\label{sund}
 Let $(A, +, \circ )$ be a  strongly  nilpotent  brace over the field of rational numbers $\mathbb Q$. Let  $a,b\in A$. Denote 
\[a\cdot b=lim _{n\rightarrow \infty }2^{n}({\frac 1{2^{n}}} a)*b.\]
 Then, for  $\alpha, \gamma  \in \mathbb Q$ and $a, b,c\in A$ we have
\[(\alpha a +\gamma b)\cdot c=\alpha (a\cdot c)+ \gamma (b\cdot c)\]
and 
\[a\cdot (\alpha b+\gamma c)=\alpha (a\cdot b)+ \gamma (a\cdot c).\]
\end{lemma}
\begin{proof} 
By the definition of a left $\mathbb Q$-brace we immediately get that  $(a\cdot (\alpha b+ \gamma c))=\alpha (a\cdot b)+ \gamma (a\cdot c)$. 

 It remains to show that  for $a, b,c\in A$ we have $a\cdot (b+c)=a\cdot b+a\cdot c$  because the base field is $\mathbb Q$. Indeed,  then for $p,q\in \mathbb Z, q\neq 0$ we have 
$ ({\frac pq} a)\cdot b=p((({\frac 1p})a\cdot) b)={\frac pq}a\cdot b$.

 Observe also that by Lemma \ref{fajny} $(a+b)\cdot c=lim _{n\rightarrow \infty }
2^{n}  ({\frac 1{2^{n}}} a+{\frac 1{2^{n}}}b)* c=
lim _{n\rightarrow \infty }{{2^{n}}}({\frac 1{2^{n}}} a)* c+2^{n}({\frac 1{2^{n}}}b)*c+2^{n}C(n)$ 
 where $C(n)$ is a sum of some products of elements ${\frac 1{2^{n}}}a$ and ${\frac 1{2^{n}}}b$ and an element $c$ at the end
(because $A$ is a strongly nilpotent brace). Moreover, each product has degree at least $3$.

We need to show that $lim _{n\rightarrow \infty }2^{n}C(n)=0$.
  We may consider a vector $V_{2^{-n}a, 2^{-n}b, c}$ obtained as in Notation $2$ from products of elements $2^{-n}a$, $2^{-n}b, c$. 

 Using similar methods as in the proof of Lemma \ref{newest} we can show that for an appropriate upper-triangular matrix $M$
 with diagonal entries smaller than $\frac 12$ (equal to $2^{-i}$ for some $i>1$) we have 
$V_{{\frac 1{2^{n}}}a, {\frac 1{2^{n}}}b, c}=M^{n}V_{a,b, c}$, hence the limit $2^{n}V_{2^{-n}a, 2^{-n}b,c}$ exists and is equal zero, which implies that 
$lim _{n\rightarrow \infty }2^{n}C(n)=0$.
\end{proof} 

$ $

\section{Passage from braces to pre-Lie algebras}\label{4}

 We now explain how to obtain a pre-Lie algebra from a brace.  It is the same pre-Lie algebra as in Rump's correspondence,  but we developed an algebraic method to obtain this algebra  instead of geometric methods used by Rump.
 In a brace $(A, +, \circ)$ we will denote as usual  $a*b=a\circ b-a-b$.

\begin{theorem}\label{main}
 Let $(A, +, \circ )$ be a  strongly  nilpotent  brace over the field of rational numbers.  
For $a,b\in A$ define \[a\cdot b=lim _{n\rightarrow \infty }2^{n}({\frac 1{2^{n}}} a)*b.\]
 Then $(A, +, \cdot )$ is a pre-Lie algebra. 
\end{theorem}
\begin{proof}
 Observe first that that $lim _{n\rightarrow \infty }2^{n}({\frac 1{2^{n}}} a)*b$ exists by Lemma \ref{newest}.

 We will now show that for every $a,b,c\in A$ we have
\[a\cdot (b\cdot c)-(a\cdot b)\cdot c=b\cdot (a\cdot c)-(b\cdot a)\cdot c.\]

  By Lemma \ref{fajny} we get 
\[(x+y)*z=x*z+y*z +x*(y*z)-(x*y)*z +d(x,y,z),\]
\[ (y+x)*z=x*z+y*z+y*(x*z)-(y*z)*z +d(y,x, z),\]
where $d(x, y, z)=E^{T}V_{x, y, z}$ for some vector $E$ which does not depend of $x, y, z$, and where $V_{x, y, z}$ is as in Notation $2$ (moreover $d(x, y, z)$ is a combination of elements with at least 3 occurences of elements from the set $\{x, y\}$). 
 It follows that 
\[x*(y*z)-(x*y)*z-y*(x*z)+(y*x)*z=d(y, x, z)-d(x, y, z).\]

Let $a,b,c\in A$ and let $m, n$ be natural numbers.
  Applying it to $x={\frac 1{2^{n}}}a$, $y={\frac 1{2^{m}}}b$, $z=c$ we get 
\[({\frac 1{2^{n}}}a)*(({\frac 1{2^{m}}}b)*c)-(({\frac 1{2^{n}}}a)*({\frac 1{2^{m}}}b))*c+d({\frac 1{2^{n}}}a,{\frac 1{2^{m}}}b,c)=\]
\[ =({\frac 1{2^{m}}}b)*(({\frac 1{2^{n}}}a)*c)-(({\frac 1{2^{m}}}b)*({\frac 1{2^{n}}}a))*c+d({\frac 1{2^{m}}}b,{\frac 1{2^{n}}}a, c).\]

 By using the formula from Lemma \ref{newest} we obtain 

 \[lim_{m\rightarrow \infty }lim _{n\rightarrow \infty }2^{m+n}[({\frac 1{2^{n}}}a)*(({\frac 1{2^{m}}}b)*c)-(({\frac 1{2^{n}}}a)*({\frac 1{2^{m}}}b))*c+d({\frac 1{2^{n}}}a,{\frac 1{2^{m}}}b,c)]=\]
\[=a\cdot (b\cdot c)-(a\cdot b)\cdot c\]
and 
 \[lim_{m\rightarrow \infty }lim _{n\rightarrow \infty }2^{m+n}[({\frac 1{2^{m}}}b)*(({\frac 1{2^{n}}}a)*c)-(({\frac 1{2^{m}}}b)*({\frac 1{2^{n}}}a))*c+d({\frac 1{2^{m}}},{\frac 1{2^{n}}}a, c)]=\]
\[b\cdot (a\cdot c)-(b\cdot a)\cdot c.\]

 Consequently 
\[a\cdot (b\cdot c)-(a\cdot b)\cdot c=b\cdot (a\cdot c)-(b\cdot a)\cdot c.\]

 It remains to show that for $\alpha \in \mathbb Q$ we have
$(\alpha a)\cdot b=\alpha (a\cdot b)$ and $a\cdot (\alpha b)=\alpha (a\cdot b)$. It follows from Lemma \ref{sund}.

\end{proof}

\section{The correspondence is one-to-one}\label{fasola}

$ $

 In this chapter we show that the correspondence between strongly nilpotent $\mathbb F$-braces and pre-Lie algebras over $\mathbb F$ is one-to-one for $\mathbb F=\mathbb Q$.  Recall that $\mathbb N$ denotes the set of natural numbers. We start with the following.

\begin{proposition} \label{lim}
 Let $(A, +, \cdot )$ be a nilpotent  pre-Lie algebra over a field $\mathbb F$ of characteristic zero, and let $(A, +, \circ )$ be the  brace obtained as in Section \ref{mi}, so $(A, \circ )$ is the formal group of flows of the pre-Lie algebra $A$. Suppose that  $\mathbb  F=\mathbb R$ the field of real numbers (or the field of rational numbers).
 Then for every $a\in A$ there exists  limit 
\[ lim_{n\rightarrow \infty}2^{n} ({\frac 1{2^{n}}}a)*b.\]
 Moreover
\[ a\cdot b=lim_{n\rightarrow \infty}2^{n} ({\frac 1{2^{n}}}a)*b,\]
 where $n\in \mathbb N$.
\end{proposition}
\begin{proof}
 It follows immediately from the fact that the multiplication in a pre-Lie algebra is bilinear, and from the fact that $a* b$ can be expressed as in Theorem \ref{1}.
\end{proof}

We now obtain the `reverse' theorem to Theorem \ref{1}:

\begin{theorem}\label{5}
 Let $(A, +, \circ )$ be a brace and let $(A, +, \cdot )$ be a nilpotent  pre-Lie algebra over the field $\mathbb Q$ 
 obtained from this brace using Theorem \ref{main}, so $a\cdot b=lim_{n\rightarrow \infty}2^{n} ({\frac 1{2^{n}}}a)*b.$
 
 Then $(A, \circ)$ is the group of flows of the pre-Lie algebra $A$, and  $(A, +,  \circ )$ can be obtained as in 
 Section \ref{mi} from pre-Lie algebra $(A, +, \cdot )$.
\end{theorem}
\begin{proof} Let $E_{a,b}$ be as in Notation $1$. 
 Observe that by Lemma \ref{newest} applied several times 
\[a\cdot  b=a* b +\sum_{w\in E_{a,b}}\alpha _{w}w\] where 
$\alpha _{w}\in F$ do not depend on $w$ and each $w$ is a product of at least $3$ elements from the set $\{a, b\}$. Observe that coefficients $\alpha _{w}$ do not  depend on the brace $A$, as they were constructed using the formula from Lemma \ref{fajny} which holds in every strongly nilpotent brace (as we can consider $V_{a, b}$ to be an infinite vector with almost all entries zero in $A$).
 Therefore $a*b=a\cdot b-\sum_{w\in E_{a,b}}\alpha _{w}w$, and now we can use this formula several times  to write every element from $E_{a,b}$ as a product of elements $a$ and $b$ under the operation $\cdot $.
 In this way we can recover the brace $(A,  +, \circ )$ from  the pre-Lie algebra $(A, \cdot , +)$. 

Notice that because we know that pre-Lie algebra $(A, +, \cdot )$ can be obtained as in Theorem \ref{main} from the brace which is it's group of flows (by Theorem \ref{lim})     
 we can use the same reasoning and 'recover' the group of flows using the same formula.

Therefore $(A, +, \circ)$ is the group of flows of pre-Lie algebra $A$. 
\end{proof}

By combining results from Theorems \ref{lim} and \ref{5} we get the following corollary.

\begin{corollary}
 There is one-to-one correspondence between the set of strongly nilpotent $\mathbb Q$-braces and the set  of  nilpotent pre-Lie algebras over $\mathbb Q$.
\end{corollary}

{\bf Acknowledgments.} I am  very thankful to Wolfgang Rump for  answering questions about his construction and for useful comments. This research was supported by the EPSRC grant EP/R034826/1.

\end{document}